\documentclass[12pt]{amsart}
\usepackage[all]{xy}
\usepackage{amsfonts}
\usepackage{amssymb}
\usepackage{pb-diagram}
\usepackage{enumerate}

\usepackage{color}

\newtheorem{teo}{Theorem}[section]
\newtheorem{lem}[teo]{Lemma}
\newtheorem{prop}[teo]{Proposition}

\newtheorem{cor}[teo]{Corollary}

\theoremstyle{definition}
\newtheorem{dfn}[teo]{Definition}

\newtheorem{ex}[teo]{Example}
\def\<{\langle}
\def\>{\rangle}
\def\ss{\subset}

\def\a{\alpha}
\def\b{\beta}
\def\g{\gamma}
\def\d{\delta}
\def\e{\varepsilon}
\def\l{{\lambda}}
\def\r{\rho}

\def\t{\tau}

\def\w{\omega}
\def\f{{\varphi}}

\def\D{{\Delta}}
\def\G{{\Gamma}}

\def\C{{\mathbb C}}
\def\R{{\mathbb R}}
\def\Z{{\mathbb Z}}

\def\A{{\mathcal A}}

\def\supp{\operatorname{supp}}

\def\1{\mathbf 1}

\newcommand{\ov}[1]{\overline{#1}}
\newcommand{\til}[1]{\widetilde{#1}}
\newcommand{\wh}[1]{\widehat{#1}}

\def\N{{\mathbb N}}

\begin{document}

\title[Modules from group actions]
{Hilbert $C^*$-modules from group actions: beyond the finite orbits case}

\author{Michael Frank}
\thanks{This work is a part of the joint DFG-RFBR project (RFBR grant
07-01-91555 / DFG project ''K-Theory, $C^*$-Algebras, and Index
Theory''.)}
\address{HTWK Leipzig, FB IMN, Postfach 301166, D-04251 Leipzig, Germany}
\email{mfrank@imn.htwk-leipzig.de}
\urladdr{http://www.imn.htwk-leipzig.de/\~{}mfrank}

\author{Vladimir Manuilov}
\thanks{The second named author was also partially supported by the grant HIII\hspace{-2.3ex}%
\rule{1.9ex}{0.07ex}\,-1562.2008.1}
\address{Dept. of Mech. and Math., Moscow State University,
119991 GSP-1  Moscow, Russia\\ $.\quad {\it and}$  Harbin Institute of
Technology, Harbin, P. R. China} \email{manuilov@mech.math.msu.su}
\urladdr{http://mech.math.msu.su/\~{}manuilov}

\author{Evgenij Troitsky}
\thanks{The third named author was also partially supported by the RFBR grant 07-01-00046.}
\address{Dept. of Mech. and Math., Moscow State University,
119991 GSP-1  Moscow, Russia}
\email{troitsky@mech.math.msu.su}
\urladdr{http://mech.math.msu.su/\~{}troitsky}


\begin{abstract}
Continuous actions of topological groups on compact Hausdorff spaces
$X$ are investigated which induce almost periodic functions in the
corresponding commutative $C^*$-algebra. The unique invariant mean on
the group resulting from averaging allows to derive a $C^*$-valued
inner product and a Hilbert $C^*$-module which serve as an environment
to describe characteristics of the group action. For uniformly continuous,
Lyapunov stable actions the derived invariant mean $M(\phi_x)$ is
continuous on $X$ for any element $\phi \in C(X)$, and the
induced $C^*$-valued inner product corresponds to a conditional
expectation from $C(X)$ onto the fixed point algebra of the action
defined by averaging on orbits.
In the case of selfduality of the Hilbert $C^*$-module all orbits
are shown to have the same cardinality. Stable actions on compact
metric spaces give rise to $C^*$-reflexive Hilbert $C^*$-modules.
The same is true if the cardinality of finite orbits is uniformly
bounded and the number of closures of infinite orbits is finite.
A number of examples illustrate typical situations appearing
beyond the classified cases.
\end{abstract}

\maketitle

%\tableofcontents

%%%%%%%%%%%%%%%%%%%%%%%%%%%%%%%%%%%%%%%%%%%%%%%%%%%%%%%%%%%%%%%%
\section{Introduction}

%\section{Formulation of the problem}
Investigating continuous group actions on topological spaces
several mathematical approaches may be applied. In the present
paper the authors continue their work started in
\cite{FMTZAA,TroPAMS} which relies on the Gel'fand duality of
locally compact Hausdorff spaces and commutative $C^*$-algebras.
In the dual picture some well-known results from functional
analysis and noncommutative geometry can be applied to get new
insights, often also for related noncommutative situations of
group actions on general $C^*$-algebras.

Consider a continuous action of a topological group $\Gamma$ on a
compact Hausdorff space $X$. Following the Gel'fand duality it can
be seen as a continuous action of $\Gamma$ on the commutative
$C^*$-algebra $C(X)$ of all continuous complex-valued functions on
$X$. Let us denote the subalgebra of $\Gamma$-invariant functions
on $X$ by $C_\Gamma(X)\subset C(X)$.

We wish to introduce the structure of a pre-Hilbert $C^*$-module
over $C_\G(X)$ on $C(X)$ which expresses significant properties of
the action of $\Gamma$ on $X$. One way to find suitable
$C^*$-valued inner products on $C(X)$ is the search for
conditional expectations $E: C(X) \to C_\Gamma(X)$ which are a
kind of mean over the group action of $\Gamma$ on $C(X)$ and
canonically give rise to the Hilbert $C_\Gamma(X)$-module
structures on $C(X)$ we are looking for. We followed that approach
in \cite{FMTZAA,TroPAMS} (see \cite{RieffelEpr98} for
a related discussion).

Here we want to consider a more general approach closer to the
topological background. For elements $\phi, \psi \in C(X)$ and for
the derived group maps
\begin{equation}\label{eq:innerprod1}
   \phi_x :\Gamma \to \C,\qquad \phi_x(g)=\phi(gx),
   \,\, (x \in X)
\end{equation}
we want to select a suitable normalized invariant mean $m_\Gamma$
on $\Gamma$ such that a $C_\Gamma(X)$-valued inner product on
$C(X)$ could be defined like
\begin{equation}\label{eq:innerprod}
   \<\phi,\psi\>(x) := m_\Gamma(\phi_x \overline{\psi_x}),
   \,\, (x \in X) \,.
\end{equation}
Of course, we would have to suppose $\Gamma$ to be amenable at
this point to warrant the existence of the (left) invariant mean
$m_\Gamma$. The product (\ref{eq:innerprod}) has to satisfy at
least the following two properties (with $\Gamma$-invariance
following from the definition):
\begin{enumerate}[1)]
    \item The resulting functions $m_\Gamma(\phi_x \overline{\psi_x})$
          are continuous in the argument $x \in X$.
    \item The value $\<\phi,\phi\>(x)$ is always positive,
          if $\phi(gx)\ne 0$ for some $g \in \Gamma$,
          some $x \in X$.
\end{enumerate}
One can observe that the property 2) would e.g. follow from the
following supposition:

\smallskip \noindent
2') For any $x \in X$ and any non-zero $\phi \in C(X)$ the
     map $\phi_x$ is a non-zero almost periodic function on
     $\Gamma$.

\smallskip \noindent
The supposition 2') would allow us:
\begin{itemize}
    \item to avoid the restriction on $\Gamma$ to be
          amenable,
    \item to overcome the dependence on the particular
          choice of $m_\Gamma$,
\end{itemize}
by passing from (\ref{eq:innerprod}) to
\begin{equation}\label{eq:innerprod2}
     \<\phi,\psi\>(x) := M(\phi_x \overline{\psi_x}),
\end{equation}
where the map $M:\Gamma \to C$ is the unique invariant mean on
almost periodic functions with respect to the given action of
$\Gamma$, when 1) and 2') are supposed to hold
(cf.~the appendix). The link to results in \cite{FMTZAA,TroPAMS}
is given by constructing a suitable conditional expectation
$E_\Gamma: C(X) \to C_\Gamma(X)$ by the rule

\smallskip \noindent
1') For any $\f\in C(X)$ the function
    $E_\Gamma(\phi)(x):=M(\phi_x)$ is continuous in $x$.

\smallskip \noindent
Properties 1') and 2') provide that the formula
(\ref{eq:innerprod}) makes $C(X)$ a pre-Hilbert $C^*$-module over
$C_\Gamma(X)$. Let us denote its completion by $L_\Gamma(X)$.

We are interested in two questions here:
\begin{enumerate}
\item
For which actions the conditions 1') and 2') hold?
\item
If they hold, what properties does $L_\Gamma(X)$ have?
\end{enumerate}

Our reference on almost periodic functions is \cite{DixmierC*}.
Hilbert $C^*$-modules were introduced in
\cite{RieffelInduced} and \cite{Paschke}.
For facts on Hilbert $C^*$-modules we refer the reader to
\cite{MaTroBook,Lance,RaeWil}. Recall that for a Hilbert
$C^*$-module $L$ over a $C^*$-algebra $A$ the $A$-dual module $L'$ is
the module of all bounded $A$-linear maps from $L$ to $A$. $L$ is
called self-dual (resp. $C^*$-reflexive) if $L=L'$ (resp. if $L=L''$).

\bigskip
Our paper is organized as follows: In the Section \ref{sec:uniform}
we will give some sufficient conditions for conditions 1') and 2')
to hold, hence, for the existence of $C_\Gamma(X)$-valued inner
products on the $C^*$-algebra $C(X)$. We also show that our type
of averaging is the same as the averaging over orbits.
Section \ref{sect:self-dual} deals with the more restrictive
situations in which the resulting Hilbert $C^*$-module turns out
to be self-dual. In Section \ref{sect:reflexive} we revisit the
situation of resulting $C^*$-reflexive Hilbert $C^*$-modules and
obtain an important restriction on $X$ to be supposed. In Section
\ref{sect:examples} we give some examples showing different
possible behavior of averaging. The Appendix is devoted to the
proof of the uniqueness of a measure used for averaging.

%%%%%%%%%%%%%%%%%%%%%%%%%%%%%%%%%%%%%%%%%%%%%%%%%%%%%%%%%%%%%%%%
\section{Lyapunov stability and continuity of averaging}\label{sec:uniform}

We want to find conditions under which a well-defined averaging
over the group action on orbits exists in the case of infinite
orbits. For this aim we introduce additionally to the condition of
uniform continuity discussed in \cite{FMTZAA,TroPAMS} the
condition of Lyapunov stability. The latter condition ensures
uniform continuity, the well-definedness of averaging and the
existence of a conditional expectation onto the fixed-point
algebra which gives rise to a $C^*$-valued inner product and a
resulting Hilbert $C^*$-module structure. In subsequent sections
we can apply this tool to characterize those group actions on
compact Hausdorff spaces with infinite orbits.

\begin{dfn}
 We say that an action of a group $G$ on a locally compact
 Hausdorff space $X$ is {\it uniformly continuous\/} if for every point
 $x \in X$ and every neighborhood $U_x$ of $x$ there exists a neighborhood
 $V_x$ of $x$ such that $g(V_x) \subseteq U_x$ for every $g\in G_x$,
 where $G_x$ denotes the stabilizer of $x$.
\end{dfn}

\begin{teo}
Let an action of a topological group $\G$ on a compact Hausdorff space $X$
be uniformly continuous. If all orbits are finite and if their size
is uniformly bounded then the average $M(\varphi_x)$ is continuous
with respect to $x\in X$ for any $\varphi\in C(X)$.
\end{teo}

\begin{proof}
If an orbit $\G x$ is finite then the function $\varphi_x$ on $\G$
is exactly periodic, hence
$$
M(\varphi_x)=\frac{1}{\#\G x}\sum_{gx\in \G x}\varphi(gx),
$$
so the average on $\G$ is the same as the average over orbits.
Continuity of the latter is provided by Lemma 2.11 from
\cite{FMTZAA}.
\end{proof}

Example \ref{ex:spiralalpha} below demonstrates that in the case
of presence of infinite orbits the uniform continuity is not
sufficient for the continuity of the average.

Now we generalize
an approach of \cite{TroPAMS} and introduce a condition which is sufficient
to overcome these difficulties.

Let $\Phi$ be a uniform structure on a compact space $X$. Recall from
\cite{Bourbaki-GeneralTopology} that, on a compact space, there is
a unique uniform structure compatible with its topology. It
consists of \emph{all} neighborhoods of the diagonal in $X\times
X$ \cite[Ch.~II, Sect.~4, Theorem~1]{Bourbaki-GeneralTopology}. If
$X$ is a metric space with a metric $d$ then the uniform structure
is the set of the neighborhoods of the diagonal $\Delta\subset
X\times X$ of the form $\{(x,y):x,y\in X, d(x,y)<\varepsilon\}$,
$\varepsilon\in(0,\infty)$.

\begin{dfn}
An action of a group $\Gamma$ on a topological space $X$ with a
uniform structure $\Phi$ compatible with its topology is called
{\it Lyapunov stable} if for any $\mathbb U\in\Phi$ and any
$x\in X$ there is $\mathbb V\in\Phi$ such that $(gx,gy)\in \mathbb
U$ for any $g\in\Gamma$ if $(x,y)\in \mathbb V$.
\end{dfn}

Note that in the case of a {\it metric} space, this definition
takes the following form:

 \begin{dfn}
\rm An action of a group $\G$ on a metric space $X$ is called
{\em Lyapunov stable\/} if for any $\e>0$ and any $x\in X$ there
exist $\d>0$ such that
 $$
\r(gx,gy)<\e\quad\mbox{for any}\quad g\in\G\quad{\rm if}\quad
\r(x,y)<\d.
 $$
 \end{dfn}

\begin{lem}
If an action of a discrete group $\G$ on a topological space $X$
with a uniform structure is Lyapunov stable then it is uniformly
continuous.
\end{lem}

\begin{proof}
For $x\in X$ and for $\mathbb U\in\Phi$ set $\mathbb U(x):=\{y\in
X:(x,y)\in \mathbb U\}$. If $W$ is a neighborhood of $x$ then
there is $\mathbb U\in\Phi$ such that $\mathbb U(x)\subset W$. By
stability, there is $\mathbb V\in\Phi$ such that $(gx,gy)\in
\mathbb U$ for any $g\in\G$ if $(x,y)\in \mathbb V$. Now let
$g\in\G x$. Take any $y\in \mathbb V(x)$. Then $(x,gy)\in \mathbb
U$, hence $gy\in \mathbb U(x)\subset W$, i.e. $g(\mathbb
V(x))\subset W$ for any $g\in\G x$.
\end{proof}

In the case when all orbits are finite, uniform continuity is
equivalent to Lyapunov stability:
\begin{prop}
Let a discrete group act uniformly continuously on a compact
Hausdorff space $X$ and let all the orbits are finite. Then this
action is Lyapunov stable.
\end{prop}
\begin{proof}
Take a neighborhood $\mathbb W$ of the diagonal in $X\times X$ and
take a point $x\in X$. Since its orbit is finite, we can take a
finite set $\{g_1,\ldots,g_s\}$ of elements in $\G$ such that
$\{g_1x,\ldots,g_sx\}$ is the orbit $\G x$. Now find a
neighborhood $U^x$ of $x$ such that $g_i(U^x)\times
g_i(U^x)\subset \mathbb W$ for each $i=1,\ldots, s$.

Uniform continuity implies that there exists a neighborhood $V^x$
of $x$ such that $hy\in U^x$ for any $y\in V^x$ and any $h\in \G
x$. Since any $g\in \G$ can be written as $g=g_ih$ for some
$i=1,\ldots,s$ and for some $h\in\G x$, $g(V^x)=g_i(h(V^x))\subset
g_i(U^x)$.

It follows from compactness of $X$ that there is a finite number
of points $x_1,\ldots,x_r$ in $X$ such that the sets
$V^{x_1},\ldots,V^{x_r}$ form a finite covering for $X$. Then
$\mathbb W_0=V^{x_1}\times V^{x_1}\cup\ldots\cup V^{x_r}\times
V^{x_r}$ is a neighborhood of the diagonal in $X\times X$.

Take $(y,z)\in \mathbb W_0$. Then there is some $1\leq j\leq r$
such that $(y,z)\in V^{x_j}\times V^{x_j}$. Then $(gy,gz)\in
g_i(U^{x_j})\times g_i(U^{x_j})$ for some $i$. By construction,
$g_i(U^{x_j})\times g_i(U^{x_j})\subset \mathbb W$, so we conclude
that $(gy,gz)\in \mathbb W$ for any $g\in\G$ whenever $(y,z)\in
\mathbb W_0$.
\end{proof}

\begin{prop}\label{prop:almperiod}
Let a discrete group $\Gamma$ act {\rm Lyapunov stably} on a
compact Hausdorff space $X$ and let $\varphi:X\to\mathbb C$ be a
continuous function. Then, for any $x\in X$, the function
$\varphi_x:\Gamma\to \mathbb C$, $\varphi_x(g):=\varphi(gx)$, is
almost periodic.

\end{prop}
\begin{proof}
Take $\varepsilon>0$. Then compactness of $X$ implies existence of
some $\mathbb U\in\Phi$ such that
\begin{equation}\label{m1}
|\varphi(x)-\varphi(y)|<\varepsilon \ \ {\rm if}\ \  (x,y)\in
\mathbb U,
\end{equation}
since $\f$ is uniformly continuous and $\Phi$ consists of all
neighborhoods of the diagonal.

Stability implies that there is $\mathbb V\in\Phi$ such that
$(gx,gy)\in \mathbb U$ for any $g$ if $(x,y)\in\mathbb V$.

Denote the closure of the orbit $\Gamma x$ by $Y\subset X$. It is
a compact subset. Compactness of $Y$ implies that one can find in
$Y$ a finite number of points of the form $g_ix$, $i=1,\ldots,s$,
such that for any $p\in Y$ there is some $i$, for which
$(g_ix,p)\in \mathbb V$ and, therefore, for any $g,h\in\Gamma$,
\begin{equation}\label{m2}
(hg_ix,hgx)\in \mathbb U \, .
\end{equation}

Then the functions $L_{g_i}\varphi_x$, $i=1,\ldots,s$, form an
$\varepsilon$-net for the set $\{L_g\varphi_x : g\in\Gamma\}$, with
respect to the uniform norm. Indeed, for any $g\in\Gamma$, there is
an index $i$ such that (\ref{m2}) holds. Then, by (\ref{m1}), we have
$$
\sup_{h\in\Gamma}|(L_g\varphi_x)(h)-(L_{g_i}\varphi_x)(h)|
=\sup_{h\in\Gamma}|\varphi(hgx)-\varphi(hg_ix)|<\varepsilon.
$$

\end{proof}

So, under the conditions of Proposition \ref{prop:almperiod} the
invariant mean $M(\varphi_x)$ is well-defined on $C(X)$.

 \begin{teo}
 \label{teo:avercorrlyap}
 Let a discrete group $\G$ act on a compact Hausdorff space $X$.
 If the action is Lyapunov stable, then the conditional expectation
 $E_\G: C(X) \to C_\Gamma(X)$ defined by $E_\Gamma(\phi)(x)=M(\phi_x)$
 is well-defined, i.e. the conditions 1') and 2') hold.
 \end{teo}

\begin{proof}
By Proposition~\ref{prop:almperiod} we only need to verify the
continuity of the mean $M(\f_x)$ with respect to $x\in X$. Let
$x\in X$ and $\e>0$ be arbitrary. Let us remind (cf.~\cite[pp.
250--251]{HewRos}) that we can choose
 $h_1,\dots,h_p\in \G$ in such a way that the uniform distance on $\G\times\G$
between the function
 $$
\frac 1p \sum_{j=1}^p D_{h_j}\f_x : \G\times\G\to\C \qquad ({\rm
where}\quad (D_h \psi)(g_1,g_2):=\psi(g_1hg_2))
 $$
and some constant is less then $\e$. In this case the uniform
distance satisfies the inequality
 $$
\left\|M(\f_x)-\frac 1p \sum_{j=1}^p D_{h_j}\f_x  \right\|_u <
2\e.
 $$
Let us choose a neighborhood $\mathbb V\in\Phi$ such that
 $$
|\f(gy)-\f(gx)|<\e,\qquad \mbox{for any }g\in\G,\:(x,y)\in \mathbb
V.
 $$
This neighborhood can be found  as in the proof of
Proposition~\ref{prop:almperiod}: first we can find a neighborhood
$\mathbb U$ such that $|\f(y)-\f(z)|<\e$ whenever $(y,z)\in
\mathbb U$ (using compactness of $X$). Then, by stability of the
action, we can find for this $\mathbb U$ another neighborhood
$\mathbb V\in\Phi$ such that $(gy,gx)\in \mathbb U$ for any
$g\in\G$ whenever $(y,x)\in \mathbb V$.

Then for any $y\in \mathbb U(x)$ one has
\begin{eqnarray*}
&&\left(\frac 1p \sum_{j=1}^p D_{h_j}\f_y\right) (g_1,g_2)=\\
&&\qquad = \left(\frac 1p \sum_{j=1}^p D_{h_j}\f_x\right)\,
(g_1,g_2)+ \frac 1p \sum_{j=1}^p\left( \f_y(g_1 h_j g_2)-\f_x(g_1
h_j g_2)\right)
\\
 &&\qquad
= \left(\frac 1p \sum_{j=1}^p D_{h_j}\f_x\right)\, (g_1,g_2)+
\frac 1p \sum_{j=1}^p\left( \f(g_1 h_j g_2y)-\f(g_1 h_j g_2
x)\right).
\end{eqnarray*}
Each term of the second summand is less then $\e$. Hence, the
second summand is less then $\e$. Thus,
 $$
\left\|M(\f_x)-\frac 1p \sum_{j=1}^p D_{h_j}\f_y  \right\|_u <
3\e
 $$
for any $y \in \mathbb U(x)$.
Therefore, considering $M(\f_x)$ as an arbitrary constant, we have
 $$
\left\|M(\f_y)-\frac 1p \sum_{j=1}^p D_{h_j}\f_y  \right\|_u <
6\e,
 $$
and finally,
 $$
\left|M(\f_y)- M(\f_x) \right| < 9\e
 $$
for any $y \in \mathbb U(x)$. Consequently, $E_\Gamma(\phi)$ is
$\Gamma$-invariant and continuous on $X$.
\end{proof}

For $x\in X$ let us denote its orbit $\Gamma x$ by $\gamma$ and
the closure of the orbit $\gamma$ in $X$ by $\overline{\gamma}$.

\begin{teo}\label{equiv.aver}
Let a discrete group $\G$ act on a compact Hausdorff space $X$.
\begin{enumerate}[\rm 1)]
  \item For a Lyapunov stable action and for the unique
    invariant mean $M: \Gamma \to \mathbb C$ we have the equality
    \begin{equation}\label{eq:twoaverages}
    M(\f_x)=\int_{\ov\g} \f|_{\ov\g}\,d\mu_{\ov\g},
    \end{equation}
    where $x\in\g$, $\mu_{\ov\g}$ is a (unique) invariant
    measure on ${\ov\g}$ of total mass 1.
  \item If $\g$ is finite, then $M(\f_x)$, $x\in\g$,
    can be taken as the standard average, as it was considered in
    {\rm \cite{FMTZAA,TroPAMS}}.
\end{enumerate}
\end{teo}

\begin{proof}
Evidently, 2) follows from 1).

Let us show that for a $\f \in C(X)$ the left-hand side of
(\ref{eq:twoaverages}) does not depend on $x\in\ov\g$. First,
evidently, it does not depend on the choice of $x$ inside the same
orbit. Hence, it is sufficient to verify that the value is the
same for $g x_2$ sufficiently close to any $x_1$ for $x_1, x_2 \in
\overline{\gamma}$ to demonstrate the invariantness with respect
to the action of $\Gamma$. By the Lyapunov
stability property, for any $\e>0$ we can find an element
$g_\e\in\G$ such that $g_\e x_2$ is so close to $x_1$ that
$|\f(g\,x_1)-\f(g\,g_\e x_2)|<\e$ for any $g\in \G$. Then
\begin{eqnarray*}
  |M(\f_{x_1})-M(\f_{x_2})| &=& |M(\f_{x_1})-M(\f_{g_\e x_2})|= |M(\f_{x_1}-\f_{g_\e x_2})| \\
    &\le & \sup_{g\in\G} |\f(g\,x_1)-\f(g\,g_\e x_2)| <\e.
\end{eqnarray*}
Since $\e$ is arbitrary small, $M(\f_{x_1})=M(\f_{x_2})$ and the
value is constant on the closure of orbits.

A similar estimation implies the continuity of this (well-defined
by the above argument) functional $m:C(\ov\g)\to\C$,
$m(\phi)=M(\phi_x)$ for $x \in \overline{\gamma}$, with respect to
the variation of closures of orbits. By the Riesz-Markov-Kakutani
theorem \cite[Theorem~3, Sect.~IV.6]{DanSch}, $m$ has the form
$$
m(f)=\int_{\ov\g} f\,d\mu,
$$
where $\mu$ is some regular countably additive complex measure on
$\ov\g$. Evidently, $\mu$ is invariant. It remains to explain why
$\mu$ ia unique. In fact, this follows from \cite[Ch.~VII, \S~1,
Problem~14]{Bourb.Int.II}. We give details in the Appendix.
\end{proof}

%%%%%%%%%%%%%%%%%%%%%%%%%%%%%%%%%%%%%%%%%%%%%%%%%%%%%%%%%%%%%%%
\section{Self-duality}\label{sect:self-dual}

After a characterization of the inner structure of Hilbert
$C^*$-modules that arise from Lyapunov stable actions we are
going to describe the interrelation between certain properties of
the action and self-daulity of the resulting Hilbert $C^*$-module.

\begin{lem}\label{lem:LyapunovGelfand}
Let a discrete group $\G$ act on a compact Hausdorff space $X$.
If the action is Lyapunov stable then any two orbits are either
separated from each other, or have the same closure. Thus,
closures of orbits are separated sets in $X/\G$. The Gelfand
spectrum of $C_\G(X)$ is the set of closures of orbits.
\end{lem}

\begin{proof}
Suppose, two orbits $\g=\G x$ and $\g'=\G y$ are not separated but
have distinct closures. This means (after a shift, if necessary)
that $y\in \ov{\g}$, but $g_0 y\not\in \ov{\g}$ for some
$g_0\in\G$. Then there exists some $\mathbb U$ of the uniform
structure $\Phi$ such that there are no points of the form $(g_0 y,g_1
x)$ in $\mathbb U$. Let us take a neighborhood $\mathbb V\in\Phi$
corresponding to $\mathbb U$ by the definition of Lyapunov
stability. Take $(y,g_2 x)\in \mathbb V$. Then $(g_0 y, g_0 g_2
x)\in \mathbb U$. Take $g_1=g_0 g_2$. A contradiction.

Thus, the quotient space of closures of orbits is Hausdorff and,
hence, it coincides with the Gelfand spectrum of $C_\Gamma(X)$.
\end{proof}

\begin{teo}\label{teo:opisanieLgamma}
Let a discrete group $\G$ act on a compact Hausdorff space $X$.
In the case of a Lyapunov stable action the module $L_\G(X)$ has
the following description: it consists of all functions $\psi:X\to
\C$ such that
\begin{enumerate}[\rm 1)]
\item $\psi|_{\ov \g}\in L_2(\ov \g,\mu_{\ov \g})$, where
$\mu_{\ov \g}$ is a unique normalized invariant measure on ${\ov \g}$
for any orbit $\g$,
\item for any $\f\in C(X)$ the function $\<\psi,\f\>_L$ is continuous.
\end{enumerate}

\noindent In particular, the average $\<\psi, \mathbf 1\>_L$ of
such a function $\psi$ is continuous on $X/\G$.
\end{teo}

\begin{proof}
We should prove the following two facts: a) the set of continuous
functions on $X$ satisfies these conditions, and b) it is dense in this
set with respect to the $C^*$-valued inner product on the module $L_\G(X)$.

Condition 1) of the assertion above should be interpreted via the
equality
\begin{equation}\label{eq:closure-and-orbit}
    \int_{\ov{\g}} \psi|_{\ov \g}\,d\mu_{\ov \g} =M(\psi_x),\qquad \psi\in C(X),\quad x\in\g.
\end{equation}
This
equality follows from two facts: i) the left-hand side depends
only on $\psi|_\g$ and defines an invariant mean on almost
periodic functions on $\G$, ii) such a mean is unique. (see Theorem \ref{equiv.aver})

Thus, by the results of Proposition \ref{prop:almperiod} and
Theorem \ref{teo:avercorrlyap} of Section \ref{sec:uniform},
condition a) is fulfilled.

Now take an arbitrary function $\psi(x)$ satisfying the conditions 1) and 2)
of the assertion above, an arbitrary function  $\f\in C(X)$ with $\|\f\|_L\le 1$,
and an arbitrary small $\e>0$. Consider the closure of
an orbit $\ov{\g}$. Choose a continuous function $f_{\ov{\g}}: \ov{\g}\to\C$
such that
\begin{equation}\label{eq:closephigamma}
    \int_{\ov{\g}} \left|\psi|_{\ov \g}-f_{\ov{\g}}\right|^2\,d\mu < \e^2.
\end{equation}
By normality of $X$, $f_{\ov{\g}}$ can be extended to a continuous
function $\wh f_{\ov{\g}}:X\to \C$. There exists a neighborhood
$U_{\ov{\g}}$ of $\ov{\g}$ in the Gelfand spectrum $\til{X/\G}$ of
$C_\G(X)$, such that
\begin{equation}\label{eq:sravnvokr}
     \int_{\ov{\g''}} (\psi|_{\ov {\g''}}-\wh f_{\ov{\g}}|_{\ov{\g''}})
     \ov{\f}_{\ov {\g''}}\,d\mu
     < 2\e, \qquad {\ov {\g''}}\in U_{\ov{\g}}.
\end{equation}
This follows from (\ref{eq:closephigamma}) and 2). Choose a finite
subcovering $U_{\ov{\g_i}}$ of $\til{X/\G}$ and a subordinated partition
of unity $\w_i$, $i=1,\dots,I$.
This can be done by Lemma \ref{lem:LyapunovGelfand}.
Then $\sup|\<\psi-f,\f\>|\le 2\e$,
where
$$
f=\sum_{i=1}^I \w_i \wh f_{\ov{\g_i}}.
$$
Thus $f\in C(X)$ is $2\e$-close to $\psi$.
\end{proof}

\begin{teo}\label{teo:infiniteselfdual}
Let a discrete group $\G$ act on a compact Hausdorff space $X$.
\begin{enumerate}[\rm 1)]
\item Suppose, the module $L_\G(X)$ is self-dual
and the Gelfand spectrum $\til{X/\G}$ of the algebra
of continuous
invariant functions $C_\G(X)$ has no isolated points. Then
there are only finitely many $\ov\g$ with infinite $\g$
and all finite orbits have the same cardinality.

\item If there are only finitely many $\ov\g$ with infinite $\g$
and all finite orbits have the same cardinality,
then the module $L_\G(X)$ is self-dual.
\end{enumerate}
\end{teo}

\begin{proof}
1) By \cite{PavTro}, the restriction on the Gelfand spectrum
implies that $L_\G(X)$ is finitely generated projective. Let $N$
be the cardinality of some of its generator systems. Thus, the
number of points of each finite orbit is $\le N$. This follows
from the epimorphity of the restriction map $L_\G(X)\to L_\G(Y)$,
where $Y\ss X$ is a closed $\G$-invariant set. Indeed, $Y$ is a
closed set in a normal space, hence continuous functions on it are
extendable by the Tietze theorem.

In this situation of uniform boundness of the cardinality of
finite orbits, the subset $X_f$, formed by all finite orbits,
is a closed (invariant) subset of
$X$.\label{pagefinpart}
Indeed, suppose, an infinite orbit $\g$ is in the
closure of $X_f$. Choose a cover of $X$ by (a finite number of)
open sets $U_i$, such that no one of them is covered by the
others, and $\g$ is covered by more than $N$ of these $U_i$'s. Let
$\mathbb{U}=\cup_i (U_i\times U_i)$ be an element of the uniform
structure on $X$. Then there exists another neighborhood $\mathbb
V$ of the diagonal in $X\times X$ such that $\G(\mathbb V)\ss
\mathbb U$ under the diagonal action. Choosing a finite orbit (of
cardinality $\le N$) $\mathbb V$-close to $\g$ we obtain a
contradiction to properties of $\mathbb U$.

Thus, as above, $L_\G(X_f)$ is finitely generated. Moreover, it is
projective, because the projection associated with a canonical
isometric embedding of the finitely generated projective
$C(\til{X/\G})$-module $L_\G(X)$ into a standard finitely
generated $C(\til{X/\G})$-module $C(\til{X/\G})^N$, say $\pi:
C(\til{X/\G})^N \to L_\G(X)$, restricted to $X_f$ gives an
epimorphic idempotent mapping
$$
\pi':C(\til{X_f/\G})^N \to L_\G(X_f),
$$
defined by the restriction of matrix elements of the projection $\pi$.
Epimorphity follows from the above argument which relied on the Tietze
theorem.

We arrive to the case considered in \cite{FMTZAA} and
\cite{TroPAMS}. As it is explained in Theorem \ref{equiv.aver},
the average over finite orbits is the same as in these papers, and
the inner product is the same. Thus, $L_\G(X_f)=C(X_f)$. By the
results of \cite{FMTZAA} and \cite{TroPAMS}, under our assumptions
this module is finitely generated projective if and only if all
(finite) orbits have the same cardinality.

Now we pass to proving the statement about infinite orbits.
Suppose there exists an infinite number of closures of infinite
orbits: $\ov\g_i$, $i=1,2,\dots$.
We need to construct a $C_\G(X)$-functional on $L_\G(X)$,
which is not an element of $L_\G(X)$.

Passing to a subsequence, if necessary, we can assume that for
each point $z_i\in \til{X/\G}$ representing $\ov\g_i$, we can
choose an open neighborhood $U_i$, such that $U_i\cap U_j
=\emptyset$ if $i\ne j$. Indeed, suppose the opposite. This
implies that for one of the points, say $z_1$, and any its
neighborhood only finitely many points from the set
$\{z_1,z_2,\dots\}$ are off this neighborhood (i.e.,
$z_2,z_3,\dots \to z_0$). We choose a neighborhood $U'_1\ni z_0$,
such that there is $z_{i_1}\not\in U'_1$, and (by normality) a
neighborhood $U''_1$ of $z_0$, such that $\ov{U''_1}\ss U'_1$ and
there is a neighborhood $U_1$ of $z_{i_1}$, such that $U_1\cap
\ov{U''_1}= \varnothing$. Take $U'_2\ss U''_1$, such that there
exists $z_{i_2}\in U''_1$, $z_{i_2}\not\in U'_2$. Take, by
normality, $U''_2\ni z_{i_0}$, such that $U''_2\ss \ov{U''_2}\ss
U'_2$ and there exists a neighborhood $U_2$ of $z_{i_2}$ such that
$U_2\ss U''_1$ and $U_2\cap \ov{U''_2}=\varnothing$. And so on.
Finally, collect the points $\{ z_{i_n} \}$ which have the
required property.

Let us define $f_i:\ov\g_i\to\{0, {\sqrt i}\}$ to be the indicator
functions of subsets of $\ov\g_i$ with $\mu_i(\supp f_i)=\frac 1
i$ (where $\mu_i$ is the invariant measure on $\ov\g_i$ of total
mass 1). Thus $f_i\in L^2(\ov\g_i,\mu_i)$,
$\<f_i,f_i\>_{\ov\g_i}=\int_{\ov\g_i}|f_i|^2 \,d\mu_i=1$ and
$\int_{\ov\g_i}|f_i| \,d\mu_i=1/{\sqrt i}$.

 Choose $\a_i\in C(X)$ ($i=1,2,\dots$) such that
\begin{enumerate}[\rm 1)]
\item $\supp \a_i \ss p^{-1}(U_i)$, where $p:X\to \til{X/\G}$
is the canonical projection;
\item $\a_i(X)\ss \left[0,\frac 1{\sqrt i}\right]$,
$\a_i(\ov\g_i)= \left[0,\frac 1{\sqrt i}\right]$;
\item $\|\a_i|_{\ov\g_i}-f_i\|_{L^2} < \frac 1{2^i}$,
$\|\a_i|_{\ov\g_i}\|^2_{L^2}-\|f_i\|^2_{L^2} < \frac 1{2^i}$;
\item $\int_{\ov\g} | \a_i|_{\ov\g}| \,d\mu_{\ov\g}\le \frac 1{\sqrt i}+\frac 1{2^{i-1}}$ for any $\ov\g$;
\item $\int_{\ov\g} | \a_i|_{\ov\g}|^2 \,d\mu_{\ov\g}\le 1+ \frac 1{2^{i-1}}$ for any $\ov\g$.
\end{enumerate}
To construct these functions we first approximate $f_i$ by an
appropriate continuous function, then extend by the Tietze
theorem, and finally multiply by an appropriate partition of unity
function. More precisely, we first choose a continuous function
$\a'_i: \ov\g_i\to \left[0,1/\sqrt{i}\right]$, restricted to
satisfy properties 2 and 3. Then we extend by the Tietze theorem
$\a'_i$ to a continuous function $\a''_i: X\to
\left[0,1/\sqrt{i}\right]$. By Theorem \ref{teo:opisanieLgamma}
the functions
$$
\<\a''_i,1\>_L:\til{X/\G}\to [0,+\infty),\qquad
\<\a''_i,\a''_i\>_L:\til{X/\G}\to [0,+\infty),
$$
are continuous and
$$
\<\a''_i,1\>_L(z_i)\in\left(\frac 1{\sqrt i}-\frac 1{2^{i}},\frac 1{\sqrt i}+\frac 1{2^{i}}\right)
,\qquad
\<\a''_i,\a''_i\>_L(z_i)\in\left(1-\frac 1{2^{i}},1+\frac 1{2^{i}}\right).
$$
Choose a neighborhood $U'_i\ss U_i$ of $z_i$ such that
$$
\<\a''_i,1\>_L(z_i)\in\left(\frac 1{\sqrt i}-\frac 1{2^{i-1}},\frac 1{\sqrt i}+\frac 1{2^{i-1}}\right)
,\qquad
\<\a''_i,\a''_i\>_L(z_i)\in\left(1-\frac 1{2^{i-1}},1+\frac 1{2^{i-1}}\right).
$$
Let $\w_i:\til{X/\G}\to [0,1]$ be a continuous function with
$\w_i(z_i)=1$ and $\supp \w_i\in U'_i$. Put $\wh\w_i:=p^*\w_i:X\to
[0,1]$ and $\a_i:=\wh\w_i \a''_i$. They are the required ones.

Define a function $h:X\to [0,+\infty)$ to be equal to $\a_i$
on $p^{-1}(U_i)$ ($i=1,2,\dots$) and 0 otherwise.
First, we wish to show that $h\not\in L_\G(X)$. Indeed,
$\<h,h\>_L$ is greater than $1-\frac 1{2^i}>1/2$
at each $z_i$ and vanishes at any accumulation
point of $\{z_i\}$.

Now let us show that $h\in L_\G(X)'$. Let $\f$ be a continuous function
on $Y$ such that $\|\<\f,\f\>\|_L\le 1$. Then for any $\ov\g$ in some
$p^{-1}(U_i)$ we have (using property 5)
$$
|\<h,\f\>|_{\ov\g}|=|\<\a_i|_{\ov\g},\f|_{\ov\g}\>|\le
\<\a_i|_{\ov\g},\a_i|_{\ov\g}\>^{1/2}\cdot \<\f,\f\>^{1/2} \le 2.
$$
For the remaining $\ov\g$'s this product vanishes.
It remains to show that $\<h,\f\>$ is a continuous (invariant)
function, i.e. that for any $\e>0$ and any point $\ov\g_0$
from the closure of $\cup_i p^{-1}(U_i)$ there is an invariant
neighborhood $W$ of $\ov\g_0$ such that
$$
\int_{\ov\g} \ov{h_{\ov\g}} \cdot \f|_{\ov\g} \,d\mu_{\ov\g} <\e
$$
for any $\ov\g \in W$. Choose $W$ not intersecting with $p^{-1}(U_i)$
for $i=1,\dots,k$, where $k>\max\left(2,\left(\frac{2\sup_{x\in X} |\f(x)|}\e\right)^2\right)$.
Then (beyond the trivial cases) $\ov\g\in p^{-1}(U_i)$, $i>k$.
Let us estimate using property 4):
\begin{eqnarray*}
  \left|\int_{\ov\g} \ov{h_{\ov\g}} \cdot \f|_{\ov\g} \,d\mu_{\ov\g}\right|  & \le&
  \sup_{x\in X} |\f(x)|\cdot \int_{\ov\g } | \a_i|_{\ov\g} | \,d\mu_{\ov\g}
=\sup_{x\in X} |\f(x)|\cdot \left(\frac 1{\sqrt i}-\frac 1{2^{i-1}},\frac 1{\sqrt i}+\frac 1{2^{i-1}}\right)\\
   &<&  \sup_{x\in X} |\f(x)|\cdot \frac 2{\sqrt i}<\e
\end{eqnarray*}
for $i>k$.
Hence, the module is not self-dual.

2) As it is explained in the first part of the proof, in this case \label{pagefinpart2}
$$
X=X_f \sqcup \ov\g_1 \sqcup \dots \sqcup \ov\g_n,
$$
$$
L_\G(X)=L_\G(X_f)\oplus L^2(\ov\g_1,\mu_{\ov\g_1}) \oplus L^2(\ov\g_n,\mu_{\ov\g_n}),
$$
\begin{eqnarray*}
  (L_\G(X))'_{C_\G(X)} &=& (L_\G(X_f))'_{C_\G(X_f)}\oplus (L^2(\ov\g_1,\mu_{\ov\g_1}))'_{C_\G(\ov\g_1)}
  \oplus (L^2(\ov\g_n,\mu_{\ov\g_n}))'_{C_\G(\ov\g_n)} \\
   &=&(L_\G(X_f))'_{C_\G(X_f)}\oplus (l^2(\C))'_\C \oplus\dots \oplus (l^2(\C))'_\C\\
   &=&(L_\G(X_f))'_{C_\G(X_f)}\oplus L^2(\ov\g_1,\mu_{\ov\g_1}) \oplus L^2(\ov\g_n,\mu_{\ov\g_n}).
\end{eqnarray*}
As it was explained above $L_\G(X_f)=C(X_f)$ in this case, and $(C(X_f))'_{C_\G(X_f)}=C(X_f)$.
\end{proof}

\begin{ex}
Let $Y$ be the cone given by the equation $x^2+y^2=z^2$, $Z\subset
Y$ be the subset of all points with $z\in
J=\{0,1,1/2,1/3,\ldots\}$. Then $Z$ is an infinite collection of
circles with one limit point $(0,0,0)$ added. Let $X$ be a union of three
distinct copies of $Z$. To describe an action of $\Z$ on $Z$ number
the circles in the double-cone consecutively by numbers of $\Z$
where the number zero is fixed to the point $(0,0,0)$.
Consider the discrete group $\G=\Z\oplus\Z_3$, where $\Z$ acts on
each circle by an irrational rotation by an angle $\a_i$
$(i=1,2,\dots)$, where $\a_i\to 0$, and where $\Z_3$ transposes the cones.
Then the module $L_\G(X)$ is not self-dual since the orbits are all
infinite except for the fixed-point.
\end{ex}

%%%%%%%%%%%%%%%%%%%%%%%%%%%%%%%%%%%%%%%%%%%%%%%%%%%%%%%%%%%%%%%%%%%%%%
\section{$C^*$-Reflexivity}\label{sect:reflexive}
\section*{4.1. \ The metric case}

In this section we would like to understand, in which situations the
Hilbert $C^*$-module $L_\G(X)$ is $C^*$-reflexive over $C_\G(X)$.
Our previous partial results \cite{FMTZAA,TroPAMS} made us believe
that the Hilbert $C^*$-module $L_\G(X)$ is $C^*$-reflexive in much
more general situations beyond the finite orbit case. It turns out
that any countably generated module over a wide class of commutative
$C^*$-algebras is $C^*$-reflexive.

\begin{teo}\label{mishchenko-trofimov}
Let $X$ be a compact metric space. Then any countably generated
module over $C(X)$ is $C^*$-reflexive.
\end{teo}

\begin{proof}
The first version of a proof appeared in
\cite{Mishchenko-Trudy_MIAN}. Then Trofimov \cite{TrofimovUMN1987}
realized that the formulation in \cite{Mishchenko-Trudy_MIAN} was
too general and provided a proof for any compact $X$ with a
certain property L. While preparing this paper, we understood that
the property L of Trofimov for $X$ is the same as the property of
$X$ to be a compact Baire space. So, any compact Hausdorff space
has this property L, and $C^*$-reflexivity would have place for any
countably generated module over any unital commutative
$C^*$-algebra, which is obviously not true, e.g. for von Neumann
algebras \cite{Paschke2}. Nevertheless, the main part of
Trofimov's proof is correct. But it was overlooked that implicitly
the proof used that, for any subset $E\subset X$ and for any point
$t_0$ in the closure of $E$, there exists a sequence of points
$t_n\in E$, which converges to $t_0$. In other words, the topology
on $X$ is supposed to possess a countable base of neighborhoods at
any point of $X$. This is not true in general,
but if we restrict ourselves to the case of compact {\it metric}
spaces then this is obviously true. Under this additional
assumption, Trofimov's proof is correct.
\end{proof}

\begin{cor}\label{corollary-reflexivity}
Let $X$ be a compact metric space, and let an action of $\G$ on
$X$ be Lyapunov stable. Then the module $L_\G(X)$ is
$C^*$-reflexive.
\end{cor}

\begin{proof}
Since $X$ is metric, the module $L_\G(X)$ is countably generated
and the $C^*$-algebra $C_\G(X)$ is separable, hence its Gelfand
spectrum is metrizable.
\end{proof}

\begin{ex}
Let $D=\prod_{k-1}^\infty D_k$, where each $D_k$ is the two-points
space with the distance between the two points equal to $2^{-k}$,
and let $X=J\times D$. Let $G=\oplus_{k=1}^\infty \mathbb Z_2$,
$G_n=\oplus_{k=1}^n\mathbb Z_2$ and $\pi_n:G\to G_n$, $i_n:G_n\to
G$ be the standard projection and inclusion homomorphisms. Denote
their composition by $p_n=i_n\circ\pi_n:G\to G$. Let $\alpha$
denote the standard action of $G$ on $D$. Define an action $\beta$
of $G$ on $X$ by the formula
$$
\beta_g\left(\frac{1}{n},d\right)=\left(\frac{1}{n},\alpha_{p_n(g)}(d)\right),
n\in\mathbb N \setminus \{0\},\ \ {\rm and}\ \ \beta_g(0,d)=(0,\alpha_g(d)),\ \
{\rm where}\ \  d\in D.
$$

It is easy to see that the following properties hold for this
action:

\begin{itemize}
\item
The orbit of any point of the form $(\frac{1}{n},d)$ is finite and
consists of $2^n$ elements.
\item
The orbit of any point of the form $(0,d)$ is infinite.
\item
The action $\beta$ is continuous.
\item
The action $\beta$ is Lyapunov stable.
\end{itemize}

It follows from Corollary \ref{corollary-reflexivity} that the
module $L_\G (X)$ is $C^*$-reflexive in this example.
\end{ex}

%%%%%%%%%%%%%%%%%%%%%%%%%%%%%%%%%%%%%%%%%%%%%%%%%%%%%%%%%%%%%%
\section*{4.2. \ The non-metric case}

After we have clarified, how $C^*$-reflexivity arises in the metric
case, let us pass to the case, when $X$ is non-metric. To begin
with, we give an example of a non-$C^*$-reflexive module $L_\G(X)$.

\begin{ex}\label{ex:betaNtimesS}
Let $K$ be a (non-metrizable) compact space such that $l_2(A)$ is
not $C^*$-reflexive, where $A=C(K)$. That is the case for $A$
being a von Neumann algebra, and one of the most important cases
is that of $K=\b \N$, the Stone--\v Cech compactification of
integers. Consider the compact space $X=K\times S^1$ equipped with
the action of $\Z$ by irrational rotation in the second argument:
$$
m(y,s)=(y,e^{\a\pi m}s),\qquad m\in \Z,\quad
\a\in\R\setminus\mathbb Q,\quad s\in S^1\ss \C.
$$
This is an isometric action on $S^1$ and a trivial one on $K$,
hence it is Lyapunov stable. Evidently, the algebra of
continuous invariant functions $C_\G(X)$ is $A=C(K)$. By Theorem
\ref{teo:opisanieLgamma}, the module $L_\G(X)$ is the set of all
functions $\psi:X\to \C$ such that
\begin{enumerate}[\rm 1)]
    \item $\psi|_{\ov\g}\in L^2(\ov\g,\mu_{\ov\g})$ for each orbit $\g$, i.e.
    $\psi_y(s)=\psi(y,s)\in L^2(S^1)$;
    \item for any $\f\in C(X)$ the function $\<\psi,\f\>_L$ is continuous.
\end{enumerate}
Let $\{e_j\}$ be a countable system of orthonormal functions
forming an orthonormal basis of $L^2(S^1)$ (e.g. exponents). Then
$\{1_A\cdot e_j\}$ is an orthonormal system in $L_\G(X)$:
$$
\<1_A\cdot e_j,1_A\cdot e_k\>_L(y)=\int_{S^1}
1_A(y)e_j(s)\ov{e_k(s)}\ov{1_A(y)}\,ds =\d_{jk}.
$$
Let us show that the $C(K)$-linear span of $\{1_A\cdot e_j\}$ is dense
in $C(X)$ (hence, in $L_\G(X)$) with respect to the Hilbert module
distance. Let $\f\in C(X)$. Then for any $\e>0$ we can choose a
division $\Delta_1,\dots,\Delta_d$ of $S^1$ such that
$\sup_{\Delta_i}(\f-f_i)<\frac \e d$, where $f_i$ is independent
on $s\in S^1$, i.e. actually $f_i\in A$, and
$\sup_{\Delta_i}|f_i|\le 2 \sup_X |\f|$. Let $\chi_i$ be the
indicator function of $\Delta_i$, $i=1,\dots,d$. Take $\hat\chi_i$
to be a $\C$-linear combination of $\{e_j\}$ such that
$$
\|\chi_i - \hat\chi_i\|_{L^2(S^1)}<\frac \e d,\qquad i=1,\dots,d.
$$
Then
$$
\hat \f (y,s):= \sum_{i=1}^d f_i(y)\cdot \hat\chi_i(s) \in {\rm
span}_{C(K)} \{1_a\cdot e_j\}.
$$
Let $\psi\in C(X)$, $\|\psi\|_L \le 1$. Then
\begin{eqnarray*}
  \|\<\f-\hat\f,\psi\>\|_L &=& \sup_{y\in K} \left| \int_{S^1} (\f(y,s)-\hat\f(y,s))\ov\psi
  (y,s)\,ds\right| \\
   &\le&  \sup_{y\in K} \left| \int_{S^1} \left(\f(y,s)-
   \sum_{j=1}^d f_j(y)\chi_j(s)\right)\ov\psi (y,s)\,ds\right|\\
  &&+ \sup_{y\in K} \left| \sum_{j=1}^d \int_{S^1} f_j(y)(\chi_j(s)-\hat\chi_j(s))\ov\psi
  (y,s)\,ds\right|\\
  &\le& \sup_{y\in K} \left( \sup_{s\in S^1} \left|\f(y,s)-
   \sum_{j=1}^d f_j(y)\chi_j(s)\right|\right)\left(\int_{S^1}|\psi (y,s)|^2\,ds\right)^{1/2}\\
  &&+ \sup_{y\in K} d \cdot \sup_{j=1,\dots,d} |f_j(y)|\cdot
  \left(\int_{S^1} |\chi_j(s)-\hat\chi_j(s)|^2 \,ds\right)^{1/2}
  \cdot \left(\int_{S^1}|\psi (y,s)|^2\,ds\right)^{1/2} \\
  &\le& \sup_{i=1,\dots,d} \sup_{K\times \Delta_i}|\f(y,s)-   f_i(y)| \cdot \|\psi\|_L
  + (2\,\sup_{x\in X} \|\f(x)\|)\cdot d \cdot \frac \e d \cdot \|\psi\|_L\\
  & < & \e\,\left(1+ 2\,\sup_{x\in X} \|\f(x)\|\right).
\end{eqnarray*}
Thus, $L_\G(X)=l_2(A)$ and is not $C^*$-reflexive.
\end{ex}

Although we are far from obtaining a criterium for $C^*$-reflexivity, we
can give a sufficient condition even in the non-metric case.

\begin{teo}\label{teo:suffic_refl}
Consider a Lyapunov stable action of $\G$ on a compact Hausdorff
space $X$, where $X$ is not necessarily metrizable. Suppose, the
cardinality of finite orbits is uniformly bounded and the number of
closures of infinite orbits is finite. Then $L_\G(X)$ is $C^*$-reflexive.
\end{teo}

\begin{proof}
By the argument in the proof of the first part of
Theorem~\ref{teo:infiniteselfdual} (see
page~\pageref{pagefinpart}) finite orbits form a closed invariant
subset $X_f\ss X$. The Gelfand spectrum consists of a closed
subspace $X_f/\G$ and a finite number of isolated points
corresponding to the closures of infinite orbits. Arguing as in the
second part of Theorem~\ref{teo:infiniteselfdual} (see
page~\pageref{pagefinpart2}), we reduce this case to the case of
pure finite orbits \cite{seregin}.
\end{proof}

%%%%%%%%%%%%%%%%%%%%%%%%%%%%%%%%%%%%%%%%%%%%%%%%%%%%%%%%%%%%%%%%%%%%%%%%%%

\section{Further examples}\label{sect:examples}

We want to show by examples that there are other situations
beyond the described above in which a well-defined averaging
can be found leading to admissible $C^*$-valued inner products
and derived Hilbert $C^*$-module structures on the corresponding
commutative $C^*$-algebras.

The following example shows that we can have a non-Lyapunov stable
action with a good average.

\begin{ex}\label{ex:goodnonlyapu}
Let $\G=\Z$. Let $X$ be the direct product $X=J\times S^1$
of the subset
$$
J=\{0, 1, 1/2, 1/3, \dots \} \ss \R
$$
and the unit circle. Let $\a_i\to \a$ be a sequence of
irrational numbers, such that $\a$ is irrational and $\a/\a_i$
is irrational for every $i$. Let the generator of $\Z$
rotate $\{1/i\}\times S^1$ by $\a_i$, and the limit circle $\{0\}\times S^1$
by $\a$. Clearly we have 1') and 2') in this case.
\end{ex}

The next example demonstrates that in the case of presence of
infinite orbits, uniform continuity is not sufficient for
continuity of the average.

\begin{ex}\label{ex:spiralalpha}\cite[Example 25]{TroPAMS}
\rm Let $X\ss\R^3$ consist of two circles
 $$
S_\pm :\qquad \left\{
 \begin{array}{rcl}
x &=& \cos 2\pi t \\
y &=& \sin 2\pi t \\
z &=&  \pm 1,
 \end{array}
\right. \qquad t\in (-\infty,+\infty),
 $$
and of a non-uniform spiral
 $$
\Sigma:\qquad \left\{
 \begin{array}{rcl}
x &=& \cos 2\pi \t \\
y &=& \sin 2\pi \t \\
z &=&  \frac 2 \pi\cdot \arctan \t,
 \end{array}
\right. \qquad  (-\infty,+\infty).
 $$
Let the generator $g$ of $\G=\Z$ act on all three components by
 $$
t\mapsto t+\a,\qquad \t\mapsto \t+\a,
 $$
where $\a$ is a positive irrational number. Then the isotropy group
of each point of $X$ is trivial. Hence, the condition of uniform
continuity holds automatically.

Let $\varphi:X\to\mathbb R$ be the restriction of the function
$\mathbb R^3\ni (x,y,z)\mapsto z$ onto $X$, then the function
$\varphi_x$ on $\mathbb Z$ has the following form:
if $x\in S_\pm$ then $\varphi_x=\pm 1$; if $x\in\Sigma$ then
$\varphi_x$ is a function on $\mathbb Z$ such that
$\varphi_x(n)\in[-1,1]$ for any $n\in\mathbb Z$ and
$\lim_{n\to\pm\infty}\varphi_x(n)=\pm 1$. So, $\varphi_x$ is in
general not almost periodic and we cannot average it using our
definition. Nevertheless, we can average it using the amenability
of the group $\mathbb Z$. In this case we get
$E_\G(\varphi_x)=\left\lbrace
\begin{array}{cl}\pm 1&{\rm for}\ x\in S_{\pm}\\0&{\rm for}\
x\in\Sigma\end{array}\right.$. Thus we see that $E_\G(\varphi_x)$
is not continuous with respect to $x\in X$.
\end{ex}

\begin{ex}
In the previous example let us identify the two circles, $S_+$ and
$S_-$. Then $X$ would consist of the spiral $\Sigma$ and of the
circle $S$. Still, the function $\varphi_x$ on $\G=\mathbb Z$ need
not be almost periodic, but there is an almost periodic function
$\rho$ on $\mathbb Z$ such that for any $\varepsilon>0$ there is
finite subset $F\subset\mathbb Z$ such that
$\|\varphi_x-\rho\|<\varepsilon$ on $\mathbb Z\setminus F$. This
makes it possible to define an average $E_\G(\varphi)$ by
$E_\G(\varphi_x)=M(\rho)$. And it is easy to see that, this
time, $E_\G(\varphi_x)$ is continuous with respect to $x\in X$.
\end{ex}

\begin{ex}
Our next example is a modification of Example
\ref{ex:betaNtimesS}. Let $Y=\mathbb N\times\mathbb S^1$, $X=\beta
Y$ its Stone--\v Cech compactification. Let $\Gamma=\mathbb Z$ act
on $Y$ by rotating each circle by the irrational angle $\alpha$.
This action canonically extends to an action on $X$.

Let $s\in\mathbb S^1$. Then the inclusion $\mathbb N\to\mathbb
N\times\mathbb S^1$, $n\mapsto (n,s)$, canonically extends to a
map $s_*:\beta\mathbb N\to X$ and $m(s_*(x))=(s\cdot
e^{im\alpha})_*(x)$ for any $x\in\beta\mathbb N$ and any
$m\in\Gamma=\mathbb Z$.

Let $\varphi\in C(X)$. Since $C(X)=C_b(Y)$ (continuous functions
on $X$ are canonically identified with bounded continuous
functions on $Y$), $\varphi$ can be identified with a uniformly
bounded sequence $(\varphi^{(1)},\varphi^{(2)},\ldots)$ of
continuous functions on $\mathbb S^1$. Let $x\in\beta\mathbb N$.
Then
$$
\varphi_{s_*(x)}(m)=\varphi((s\cdot e^{im\alpha})_*(x)),
$$
where $m\in\mathbb Z$.

Let $\mathcal U_x$ be an ultrafilter on $\mathbb N$, which
corresponds to the point $x\in\beta\mathbb N$. Then
$$
\varphi_{s_*(x)}=\lim_{\mathcal
U_x}(\varphi^{(n)}(s))_{n=1}^\infty,
$$
where the limit of the sequence $(\varphi^{(n)}(s))_{n=1}^\infty$
is taken over $\mathcal U_x$, hence
$$
\varphi_{s_*(x)}(m)=\lim_{\mathcal U_x}(\varphi^{(n)}((s\cdot
e^{im\alpha})_*(x)))_{n=1}^\infty.
$$
Take $\varphi^{(n)}(s)=e^{ins}$. Then $\varphi\in C_b(Y)=C(X)$.
Then
$$
\varphi_{1_*(x)}(m)=\lim_{\mathcal
U_x}(e^{inm\alpha})_{n=1}^\infty.
$$
Let $\mathcal U$ be an ultrafilter on $\mathbb N$ such that
$\lim_{\mathcal U}(e^{in\lambda})_{n=1}^\infty=0$ for any
$\lambda\in(0,2\pi)$, and let $x_0\in\beta\mathbb N$ be the point
that corresponds to $\mathcal U$. Then
$$
\varphi_{1_*(x_0)}(m)=\left\lbrace\begin{array}{cl}1,&{\rm if}\
m=0,\\0,&{\rm if}\ m\neq 0.\end{array}\right.
$$
Thus, for the point $y=1_*(x_0)\in X$ and for the function
$\varphi\in C(X)$ we see that the function $\varphi_y$ is not
almost periodic on $\mathbb Z$.

Nevertheless, there is a `good' averaging in this example. Since
any continuous function $\varphi$ on $X$ is a uniformly bounded
sequence of functions $\varphi^{(n)}$, $n\in\mathbb N$, on
$\mathbb S^1$, it is easy to see that $C_\Gamma(X)\cong
C_b(\mathbb N)$, and one can define $E_\Gamma(\varphi)$ by the
formula $(E_\Gamma(\varphi))_n=\int_{\mathbb
S^1}\varphi^{(n)}(s)\,ds$.
\end{ex}

These examples show that a good averaging (and an inner product
with values in $C_\Gamma(X)$) can be defined in a wider class than
Lyapunov stable actions. On the other hand, as the last two
examples show, a good averaging, when exists, may give rise to a
degenerate inner product.

%%%%%%%%%%%%%%%%%%%%%%%%%%%%%%%%%%%%%%%%%%%%%%%%%%%%%%%%%%%%%%%%%
\section{Appendix}\label{sec:append}
In this section we will prove the following assertion
on the uniqueness of invariant regular measures:

\begin{lem}
Suppose, a discrete group $\G$ acts on a compact
Hausdorff space $X$ in such a way that the orbit $\g$ of an
element $a\in X$ is dense in $X$. If the action is
Lyapunov stable, then $X$ carries not
more than one invariant regular measure.
\end{lem}

In fact the assertion follows from \cite[Ch.~VII, \S~1,
Problem~14]{Bourb.Int.II}. More precisely, denote the closure
$\ov\g$ of an orbit $\g$ by $T$. We can simplify the idea of the
argument of \cite[Ch.~VII, \S~1, Problem~14]{Bourb.Int.II} because
$T$ is a compact Hausdorff space. For any subset $A\ss T$ and any
$B\ss T$ with interior points denote by $(A:B)$ the minimal
cardinality of covers of $A$ formed by sets $g B$, $g\in \G$.
Density of $\gamma$ in $T$ and Lyapunov stability imply the
finiteness of this number, or, more precisely, that such a cover
exists. To prove this, we will show that for any open set $B$ we
have $\G(B)=T$. Suppose the opposite: $x\not\in \G(B)$. Then $\G
x$ is not dense in $X$. A contradiction to Lemma
\ref{lem:LyapunovGelfand}.

In the
uniform structure of $T$ there exists a fundamental sub-system $S$
formed by invariant sets in $T\times T$ (under the diagonal action
of $\G$). Note, if $\G(\mathbb V)\ss \mathbb U$ by Lyapunov
stability, then $\cup_{g\in\G} g(\mathbb V)$ is an invariant
neighborhood of the diagonal, containing $\mathbb V$ and contained
in $\mathbb U$.

If $C\ss T$ is a third relatively compact set with interior points,
then $ (A:C)\le (A:B)\cdot (B:C). $

Recall the following notation. If $\mathbb V$ and $\mathbb W$ are subsets of $T\times T$,
then $\mathbb V\mathbb W\ss T\times T$ is formed by all pairs $(x,y)$ such that there
exists an element $z\in T$ with $(x,z)\in \mathbb V$ and $(z,y)\in\mathbb  W$.
If $K\ss T$ is an arbitrary set, then
$$
\mathbb V(K):=\{y\in T\, |\, (x,y)\in \mathbb V \quad {\rm for} \quad {\rm some}\quad x\in K \}.
$$
If $K=\{a\}$, we write $\mathbb V(a)$.

Suppose, $K\ss T$ is a compact subset, $L\supset K$ is a open set,
and $\mathbb V\in S$ is an invariant symmetric open neighborhood
of the diagonal, such that $\mathbb V(K)\ss L$, and $\mathbb W \in
S$ is a closed  invariant symmetric neighborhood of the diagonal
such that $\mathbb W\ss \mathbb V$. Let $\mathbb U$ be a symmetric
invariant set containing the diagonal such that $\mathbb U \mathbb
W\ss \mathbb V$ and $ \mathbb W \mathbb U\ss V$. Then for any
invariant (regular) positive measure $\nu$ one has
\begin{eqnarray}
  ( \mathbb W(a):\mathbb U(a))\cdot\nu(K) &\le & (L: \mathbb U(a))\cdot \nu(\mathbb V(a)),
  \label{burbint1}\\
  (K:\mathbb U(a))\cdot \nu( \mathbb  W(a)) & \le & (\mathbb V(a):\mathbb U(a))\cdot \nu(L).\label{burbint2}
\end{eqnarray}
Indeed, let $L={\bigcup}_{i=1}^{(L:\mathbb U(a))} g_i \mathbb U(a)$, $g_i\in \G$.
Then each $x\in K$ belongs at least to $( \mathbb W(a):\mathbb  U(a))$ sets from
the collection $\{g_i \mathbb V(a)\}$. Indeed, $\mathbb W(x)$ is covered by $g_i \mathbb U(a)$.
Hence, its number (number of those of them, which really intersect $\mathbb W(x)$) is
greater-equal to $(\mathbb W(x):\mathbb U(a))= (\mathbb W(a):\mathbb U(a))$.
The last equality follows from the density
of the orbit, Lyapunov property and the invariance of $\mathbb W$: $\mathbb W(gx)=g\mathbb W(x)$.
It remains to show that if $g_i \mathbb U(a)\cap\mathbb  W(x) \ne \varnothing$, then
$x\in g_i\mathbb  V(a)$. In this case let $g_i s \in g_i \mathbb U(a)\cap \mathbb W(x)$
for some $s\in T$, then $(g_i a, g_i s)\in \mathbb U$, $(x,g_i s)\in \mathbb W$.
 Since the sets are symmetric and $\mathbb  W \mathbb U\ss \mathbb V$,
we have $(x,g_i a) \in \mathbb V$. Then $(a, g_i^{-1} x)\in\mathbb  V$, $g_i^{-1} x\in \mathbb V(a)$,
$ x\in g_i\mathbb V(a)$. Thus,
$$
\nu(K) \le \frac 1{( \mathbb W(a):\mathbb  U(a))} \sum_{i=1}^{(L:\mathbb U(a))} \nu (g_i \mathbb U(a)).
$$
Since $\nu$ is invariant, we obtain (\ref{burbint1}). To obtain (\ref{burbint2})
in a similar way, we will show that each $y\in L$ belongs to not more
than $(\mathbb V(a):\mathbb U(a))$ sets of the form $g_i \mathbb W(a)$. Indeed, if
$y\in g_i\mathbb  W(a)$,
then $y= g_i s$ and $(a,s)\in \mathbb W$ for some $s\in T$, as well as $(g_i a,g_i s)$.
Let $(a,t)\in \mathbb U$. Since $\mathbb W\mathbb U\ss \mathbb V$,
$(s,t)\in \mathbb V$, $t\in \mathbb V(s)$, $\mathbb U(a)\ss \mathbb V(s)$,
$g_i \mathbb U(a)\ss g_i \mathbb V(s)=\mathbb V(y)$. So, $g_i \mathbb U(a)$
($i=1,\dots,(L:\mathbb U(a)$) form a
minimal cover of $L$ while a part of this cover is inside $\mathbb V(y)\ss L$. Thus
the cardinality of this part is lower-equal to $(\mathbb V(y):\mathbb U(a))=
(\mathbb V(a):\mathbb U(a))$. Hence,
$$
\nu(L)\ge \frac{1}{(\mathbb V(a):\mathbb  U(a))} \cdot
\sum_{i=1}^{(L:\mathbb U(a))} \nu (g_i\mathbb  W(a))
=\frac {(L:\mathbb U(a))\nu (\mathbb W(a))}{(\mathbb V(a):\mathbb  U(a))}
\ge \frac {(K:\mathbb U(a))\nu (\mathbb  W(a))}{(\mathbb V(a):\mathbb U(a))}
$$
and we obtain (\ref{burbint2}).

Suppose, $A, A_0\ss T$ are non-empty open set, $K\ss T$ is a
compact set. Choose  a fundamental system of open invariant
symmetric neighborhoods $\mathbb U_\a$ of the diagonal $\D$
indexed by a net $\A$. Put
$$
\l_\a(A):=\frac{(A:\mathbb U_\a(a))}{(A_0:\mathbb U_\a(a))},\qquad  \l(A):=\lim_{\A}\l_\a(A)
\qquad{\rm and}
\qquad \l'(K):=\inf \l(B),
$$
where $B$ runs over the set of all relatively compact open neighborhoods of $K$.
Then from (\ref{burbint1}) and (\ref{burbint2}) we obtain
\begin{equation}\label{eq:burbint3}
\l'(\mathbb W(a))\cdot \nu(K)\le \l(L)\cdot \nu(\mathbb W(a)),\qquad \l'(K) \cdot \nu(\mathbb W(a))
\le \l'(\mathbb W(a))\cdot \nu(L)
\end{equation}
for $\mathbb W$ sufficiently close to the diagonal to have $\mathbb W(K)\ss L$.
Indeed, choose open symmetric invariant neighborhoods
of $\mathbb W$ inside its sufficiently small
$\mathbb U_\a$-neighborhood: $\mathbb W_\a$ and $\mathbb V_\a$
(i.e. $\mathbb W_\a \ss \mathbb U_\a \mathbb W$, $\mathbb V_\a \ss \mathbb U_\a \mathbb W$)
such that $\ov{\mathbb W_\a}\ss \mathbb V_\a$
and still $\mathbb V_\a(K)\ss L$.
Then choose $\b_0=\b_0(\a)$ such that
$U_\b \ov{\mathbb W_\a}\ss \mathbb V_\a$,  for all $\b>\b_0$.
Take in (\ref{burbint1}) $\mathbb W=\ov{\mathbb W_\a}$, $\mathbb V=\mathbb V_\a$,
$\mathbb U(a):= \mathbb U_\b(a)$, and divide both parts
by $(A_0:\mathbb U_\b(a))$:
$$
\frac{(\ov{\mathbb W_\a}(a):\mathbb U_\b(a))}{(A_0:\mathbb U_\b(a))} \nu(K) \le
\frac{(L:\mathbb U_\b(a))}{(A_0:\mathbb U_\b(a))} \nu(\mathbb V_\a(a)),
$$
$$
\frac{({\mathbb W_\a}(a):\mathbb U_\b(a))}{(A_0:\mathbb U_\b(a))} \nu(K) \le
\frac{(L:\mathbb U_\b(a))}{(A_0:\mathbb U_\b(a))} \nu(\mathbb V_\a(a)).
$$
Passing to the limit over $\b\in\A$ we obtain
$$
\l (\mathbb W_\a(a))\cdot \nu(K) \le \l(L) \cdot \nu(\mathbb V_\a(a))
$$
for any $\a\in \A$. Passing to the limit over $\a\in\A$ we get
$$
\l' (\mathbb W(a))\cdot \nu(K) \le \l(L) \cdot \nu(\mathbb W(a)).
$$
To obtain the second inequality in (\ref{eq:burbint3}) choose a sufficiently
small open neighborhood $K_\a$ of $K$ to have $\ov{K_\a}\ss L$
and add to the above restrictions the following one: $\mathbb V_\a(\ov{K_\a})\ss L$.
Then from (\ref{burbint2}) we have
$$
\frac{(\ov{K_\a}:\mathbb U_\b(a))}{(A_0:\mathbb U_\b(a))} \nu(\mathbb W(a))\le
\frac{(V_\a(a):\mathbb U_\b(a))}{(A_0:\mathbb U_\b(a))} \nu(L),
$$
$$
\frac{( {K_\a}:\mathbb U_\b(a))}{(A_0:\mathbb U_\b(a))} \nu(\mathbb W(a))\le
\frac{(\mathbb V_\a(a):\mathbb U_\b(a))}{(A_0:\mathbb U_\b(a))} \nu(L).
$$
Passing to the limit over $\b\in\A$ the inequality
$$
\l(K_\a) \cdot \nu(\mathbb W(a))\le \l(\mathbb V_\a(a)) \cdot \nu(L)
$$
holds for any $\a\in\A$. Now passing to the limit over $\a\in\A$
we get
$$
\l'(K) \cdot \nu(\mathbb W(a))\le \l'(\mathbb W(a)) \cdot \nu(L).
$$

In particular, if $K_1\ss T$ and $K_2\ss T$ are compact subsets, and
$L_1\supset K_1$ and $L_2\supset K_2$ are relatively compact open sets,
then (\ref{eq:burbint3}) implies
$$
\frac {\nu(K_2)}{\l(L_2)}\le \frac {\nu(\mathbb W(a)}{\l'(\mathbb W(a))}
\le \frac {\nu(L_1)}{\l'(K_1)},
\qquad \l'(K_1)\cdot \nu(K_2) \le \l(L_2)\cdot \nu(L_1).
$$
Passing to the infimum over $L_2\supset K_2$ (by the definition of $\l'$)
and over $L_1\supset K_1$ (by the regularity of $\nu$) we obtain
$\l'(K_1)\cdot \nu (K_2) \le \l'(K_2)\cdot \nu (K_1)$. Transposing $K_1$
and $K_2$ we have $\l'(K_2)\cdot \nu (K_1) \le \l'(K_1) \cdot\nu (K_2)$.
Hence, $\l'(K_1)\cdot \nu (K_2) = \l'(K_2)\cdot \nu (K_1)$. This uniquely
determines $\nu$ (up to a constant multiplier).

\bibliographystyle{amsplain}
%\bibliography{COVERSeng}
\bibliography{ACT03eng}
%\end{document}
%%%%%%%%%%%%%%%%%%%%%%%%%%%%%%%%%%%%%%%

\end{document}